\documentclass[10pt]{amsart}

\usepackage{amsfonts,amsmath,amsthm}
\usepackage{amssymb}
\usepackage{latexsym}
\usepackage{epsfig}
\usepackage{graphicx,color,graphics}

 \newtheorem{lemma}{Lemma}[section]
 
 \newtheorem{theorem}[lemma]{Theorem}
 
 \newtheorem{remark}[lemma]{Remark}

\newcommand{\N}{\ifmmode{{\Bbb N}}\else{\mbox{${\Bbb N}$}}\fi}
\newcommand{\R}{\ifmmode{{\Bbb R}}\else{\mbox{${\Bbb R}$}}\fi}

\newcommand{\cal}{\mathcal}

\title[Transverse vibrations in a heated microbeam]{Stabilization of transverse vibrations of an inhomogeneous  Euler-Bernoulli beam with a thermal effect}

\author{Octavio Vera} 
\address{Department of Mathematics, Universidad del B{\'i}o-B{\'i}o, PO Box 5C, Concepci{\'o}n, Chile} 
\email{{\tt overa@ubiobio.cl}}

\author{Amelie Rambaud} 
\address{Department of Mathematics, Universidad del B{\'i}o-B{\'i}o, PO Box 5C, Concepci{\'o}n, Chile} 
\email{{\tt arambaud@ubiobio.cl}}

\author{Roberto E. Rozas} 
\address{Department of Physics,  Universidad del B{\'i}o-B{\'i}o, PO Box 5C, Concepci{\'o}n, Chile} 
\email{{\tt rrozas@ubiobio.cl}}

\date{}

\begin{document}

 \pagestyle{plain} \thispagestyle{plain}
%-Title

%+Abstract
\begin{abstract}
\noindent We consider an inhomogeneous Euler-Bernoulli (EB) beam of length $L$
clamped at both ends and subject to : an external frictional damping and a thermal effect. We prove the well-posedness of the model and analyze the behavior  of the
solution as $t \rightarrow + \infty.$ The existence is
proved using semigroup theory, and the exponential stabilization of solutions
is obtained considering multiplier technique. A numerical illustration of the energy decay is given, based on initial data close to a real physical experiment. The results generalize the ones obtained 
where there is no heat conduction taken into account (see \cite{am2eulerbernoulli-rate, am1, liu1} for EB type models, \cite{go1, go22} for wave like equations), together with the coupled systems of type damped wave equation with heat conduction \cite{go3, go4}. 
\end{abstract}
%-Abstract

%+Title
\maketitle

 Euler-Bernoulli. heat conduction. damping. Coupled system. multiplier.
\renewcommand{\theequation}{\thesection.\arabic{equation}}
\setcounter{equation}{0}

\section{Introduction}

Simplified models of beams' vibrations have interested researchers for decades. One of the main issues, both from a physical and mathematical point of view for such problems is the question of stabilization of the oscillations when time goes to infinity. There exist various models of transverse vibrations of beams. Here, we focus on a Euler-Bernoulli (EB) like model,  which is the most commonly used model because it is simple and provides reasonable engineering approximations.  See for example  \cite{di1, me1, tho1, we1, dynamics-transverse-beam1999}.  However, the standard homogeneous EB equation can be generalized in order to account for inhomogeneity of the material, heat conduction, and dampings (either viscous, or frictional). Let us also mention a model where both transverse and axial vibrations are investigated \cite{inhom-beam-stab}. 
In the present work, we investigate a model of transverse vibrations, derived from the model proposed in \cite {abouel-zenkour2015}.  Let $L >0$ be the length of a axially symmetric thin beam, along the axis $x$.  We also suppose that the microbeam is subjected to heat and to a non uniform frictional damping.  Then the transverse displacement $u$ and the temperature field $\theta$ satisfy the following set of equations (without dimension), for all $t > 0$, for all $x \ \in \ (0, L)$:
\begin{align}
\label{101}& u_{tt} + (p(x)\,u_{xx})_{xx} + 2\,q(x)\,u_{t}   + 
\kappa\,\theta_{xx} = 0,
 \\
\label{102}& \theta_{t} - \eta\,\theta_{xx} - \kappa \,u_{xxt}   = 0,
\end{align}
\noindent where $p$ is a non negative function describing the non homogeneity of the beam, and $q$, also non negative stands for the force of the frictional damping. The constants $\kappa, \, \eta$ are assumed to be strictly positive. 
 Note that system \eqref{101}-\eqref{102} is a simplified version of the one investigated in \cite{abouel-zenkour2015}. Indeed, although we added inhomogeneity and damping, the authors in \cite{abouel-zenkour2015} consider a more general law of heat conduction and assume the beam is moving along the $x$ axis at a constant velocity, and that the temperature is also damped. Moreover the coupling constants are not the same in their case (here we only have $\kappa > 0$), so their problem is more difficult to tackle.   

System \eqref{101}-\eqref{102} is complemented by Dirichlet boundary conditions
\begin{equation}
\label{103}u(0,\,t) = u(L,\,t)=0,\, \,  u_{x}(0,\,t) =
u_{x}(L,\,t)=0,\, \,  \theta(0,\,t) = \theta(L,\,t)=0,
\end{equation}
that hold for all $t > 0$, describing a \emph{clamped} beam at both ends. And finally we have initial conditions:
\begin{align}
\label{104} u(x,\,0)=u_{0}(x),\quad u_{t}(x,\,0)=u_{1}(x),\quad
\theta(x,\,0)=\theta_{0}(x),\quad x\in[0,\,L].
\end{align}

 Our purpose in this work is to provide well-posedness of system \eqref{101}-\eqref{104} under smoothness assumption on the inhomogeneity functions $p$, $q$, and to investigate the long time behavior of the energy, that is to say wether or not the energy decays exponentially.  Indeed, the study of the total energy functional  ${\cal E}:\mathbb{R}^{+}\rightarrow\mathbb{R}^{+}$, defined, in our case, for all   $t$, by
\begin{equation}
\label{201} \mathcal{E}(t) =  \frac{1}{2} \left[\int_{0}^{L}p(x)\,
u_{xx}^{2}\ dx + \int_{0}^{L}u_{t}^{2}\ dx +
\int_{0}^{L}|\theta|^{2}\ dx \right],
\end{equation}
when $t \ \rightarrow \ \infty$, answers the question of the stabilization of the vibrations of the structure. Rates of decay for the total energy have been established for various dissipative models, see for example \cite{am2eulerbernoulli-rate, chen1, chen2, ko1, ko2, lagnese1, lagnese2, lions1} and references therein.  Up to our knowledge, the exponential stability has not been investigated for our model, and provides a first answer to the stabilization of microbeams under study in \cite{abouel-zenkour2015}, with of course only the Fourier law of heat, and a non moving structure in the $x$ direction, for a start.  Moreover, our result extends the results of \cite{go1, go22} for inhomogeneous beams without heat, and the ones of \cite{go3, go4} for damped waves coupled with heat.

The rest of the paper is organized as follows. In Section \ref{sec:semigroup} we state the mathematical hypothesis and prove well-posedness of the initial boundary value problem \eqref{101}-\eqref{104} using a semigroup argument.  In Section \ref{sec:expo} we prove the exponential decay of the energy \eqref{201} which is the main result. Finally, in Section \ref{sec:numeric} we illustrate numerically the exponential decay of the energy by considering a simple Finite Difference  discretization, in the spirit of \cite{FEM-stab-beam-Zhang} and initial conditions close to a physically relevant situation, inspired by \cite{wang-xu2001}.

\renewcommand{\theequation}{\thesection.\arabic{equation}}
\setcounter{equation}{0}

\section{Well-posedness}\label{sec:semigroup}

In this section, we prove existence and uniqueness of solution to \eqref{101}-\eqref{104}.  Let us start by introducing the main assumptions and notations.

\subsection{Preliminaries}

Throughout this paper, we always
assume that $q(x)$ and $p(x)$ are positive definite functions,
$q(x)\in L^{\infty}(0,\,L),$ $p(x)\textbf{}\in H^{2}(0,\,L)$ and
that there exist constants $\alpha_{1},$ $\alpha_{2},$ $\alpha_{3}$
and $\alpha_{4}$ such that
\begin{equation}
\begin{aligned}
\label{007}0<\alpha_{1}\leq p(x)\leq
\alpha_{2},\quad\forall\,x\in[0,\,L] \\
0<\alpha_{3}\leq q(x)\leq \alpha_{4},\quad\forall\,x\in[0,\,L]
\end{aligned}
\end{equation}

Throughout this paper  $C$ is a generic constant, not necessarily
the same at each occasion(it will change from line to line), which
depends in an increasing way on the indicated quantities.

We first establish an equality describing the dissipativity property of our system.

\begin{lemma}\label{lemma:energy-equal}
For every smooth enough solution of the system \eqref{101}-\eqref{104} the total
energy ${\cal E}:\mathbb{R}^{+}\rightarrow\mathbb{R}^{+}$, defined by \eqref{201}   satisfies
\begin{equation}
\label{202} \frac{d}{dt}\mathcal{E}(t) = -2\int_{0}^{L}
q(x)\,\vert u_{t}\vert^{2}\ dx - \eta\int_{0}^{L}|\theta_{x}|^{2}\ dx.
\end{equation}
\end{lemma}
\begin{proof}
We multiply \eqref{101}, \eqref{102} by $\overline{u_{t}}$ and $\overline{\theta}$
respectively. Integrating in $x\in[0,\,L],$ using the boundary
condition \eqref{103}, summing and taking the real part, the result follows straightforward. 
\end{proof}

\begin{remark} The equations \eqref{201}-\eqref{202} can be rewritten as
\begin{equation*}
 \mathcal{E}(t) =  \frac{1}{2}
\left(\|\sqrt{p}\,u_{xx}\|_{L^{2}(0,\,L)}^{2} +
\|u_{t}\|_{L^{2}(0,\,L)}^{2} + \|\theta\|_{L^{2}(0,\,L)}^{2} \right)
\end{equation*}
and satisfies
\begin{equation*}
 \frac{d}{dt}\mathcal{E}(t) =
-\,2\,\|\sqrt{q}\,u_{t}\|_{L^{2}(0,\,L)}^{2} -
\eta\,\|\theta_{x}\|_{L^{2}(0,\,L)}^{2} .
\end{equation*}
\end{remark}

Let us now introduce the phase space associated to our problem. We 
use the standard $L^{2}(0,\,L)$ space, the scalar product
and norm are denoted by
\begin{equation*}
\langle u,\,v\rangle _{L^{2}(0,\,L)}=\int_{0}^{L}u\,\overline{v} \
dx,\qquad \|u\|_{L^{2}(0,\,L)}^{2}=\int_{0}^{L}u^{2}\ dx.
\end{equation*}
Hence, denote $\mathcal{H}=H_{0}^{2}(0,\,L)\times
L^{2}(0,\,L)\times L^{2}(0,\,L)$ endowed with the inner product
given by
\begin{eqnarray*}
\left\langle (u,\,v,\,\theta),\, (u_{1},\,v_{1},\,\theta_{1})
\right\rangle_{\mathcal{ H}}=
\int_{0}^{L}p(x)\,u_{xx}\,\overline{u}_{1xx}\ dx +
\int_{0}^{L}v\,\overline{v}_{1}\ dx + \int_{0}^{L}\theta\,
\overline{\theta}_{1}\ dx
\end{eqnarray*}
for $U=(u,\,v,\,\theta),$
$\widetilde{U}=(u_{1},\,v_{1},\,\theta_{1})$ and the norm
\begin{eqnarray*}
\|(u,\,v,\,\theta)\|_{\mathcal{H}}^{2} =
\|\sqrt{p}\,u_{xx}\|_{L^{2}(0,\,L)}^{2} + \|v\|_{L^{2}(0,\,L)}^{2} +
\|\theta\|_{L^{2}(0,\,L)}^{2}.
\end{eqnarray*}
We can easily show that the norm $\|\cdot\|_{\mathcal{H}}$ is
equivalent to the usual norm
in $\mathcal{H}$, so that the phase space $\mathcal{H}$ is a Hilbert space.

Finally, we have the Poincar\'{e} inequality
\begin{equation*}
\|u\|_{L^{2}(0,\,L)}^{2}\leq C_{p}\,\|u_{x}\|
_{L^{2}(0,\,L)}^{2},\quad \forall \;u\in \,H_{0}^{1}(0,\,L),
\end{equation*}
where $C_{p}$ is the Poincar\'{e} constant.\\

\subsection{Existence and uniqueness}

In order to establish the well-posedness, we reformulate our problem as an abstract Cauchy problem. 
Taking $u_{t}=v,$ the initial boundary value problem
\eqref{101}-\eqref{104} can be rewritten as:
\begin{equation}
\label{301}\frac{d}{dt}U(t)={\cal A}\,U(t),\quad U(0)=U_{0}, \quad
\forall \;t>0,
\end{equation}
where $\,U(t)=(u,\,v,\,\theta)^{T}\,$ and
$\,U_{0}=(u_{0},\,u_{1},\,\theta_{0})^{T}$.  The linear
operator  

$\mathcal{A}:\mathcal{D}(\mathcal{A})\subset
\mathcal{H}\rightarrow \mathcal{H}$ is given by
\begin{equation}
\label{302}\mathcal{A}\left(
\begin{array}{c}
u \\
v \\
\theta
\end{array}
\right) =\left(
\begin{array}{c}
v \\
-\,2\,q(x)\,v - (p(x)\,u_{xx})_{xx}
- \kappa\,\theta_{xx} \\
\kappa\,v_{xx} + \eta\,\theta_{xx} 
\end{array}
\right).
\end{equation}
It remains to define the domain $\mathcal{D} (\mathcal{A})$ of the operator: 
\begin{eqnarray*}
\mathcal{D}(\mathcal{A}) & = & \left\{(u,\,v,\,\theta)\in {\cal
H}:\; u\in H_{0}^{2}(0,\,L),\ -\ p(x)\,u_{xx} -
\kappa\,\theta\in H_{0}^{2}(0,\,L),\right.
\\
&  & \left.\quad \theta\in H_{0}^{2}(0,\,L),\ \kappa\,v +
\eta\,\theta\in H_{0}^{2}(0,\,L)\right\}.
\end{eqnarray*}

We are now ready to state our existence result. 

\begin{theorem}
\label{teo1} For any $U_{0}\in\mathcal{H}$ there exists a unique
solution $U(t)=(u,\,u_{t},\,\theta)$ of \eqref{101}-\eqref{104}  (or, equivalently, to the abstract Cauchy problem \eqref{301}), satisfying
\begin{align*}
& u\,\in C([0,\,\infty[;\ H_{0}^{2}(0,\,L))\cap C^{1}([0,\,\infty
[;\ L^{2}(0,\,L)) \\
& \theta\in C([0,\,\infty[;\ L^{2}(0,\,L)).
\end{align*}
However, if $\,U_{0}\in\mathcal{D}({\mathcal{A}})$ then
\begin{align*}
& u\in C^{2}([0,\,\infty[;\ H_{0}^{2}(\Omega))\cap C^{2}
([0,\,\infty[;\ L^{2}(0,\,L)) \\
&  \theta\in C([0,\,\infty[;\ H_{0}^{2}(0,\,L))\cap
C^{1}([0,\,\infty
[;\ L^{2}(0,\,L))\\
& -\ p\,u_{xx} + \kappa\,\theta,\ \kappa\,u_{t} + \eta\,\theta
\in C([0,\,\infty[;\ H_{0}^{2}(0,\,L)).
\end{align*}
\end{theorem}

\begin{proof}
Firstly, we show that the operator $\mathcal{A}$ generates a
$C_{0}$-semigroup of contractions on the space $\mathcal{H}.$
We will show that $\mathcal{A}$ is a dissipative
operator and $0$ belongs to the resolvent set of $\mathcal{A}$,
denoted by $\varrho(\mathcal{A}).$ Then our conclusion will follow
using the well known Lumer-Phillips theorem \cite{pazy}. We compute:
\begin{eqnarray*}
\lefteqn{\left\langle\mathcal{A}U,\,U\right\rangle_{\mathcal{H}} =
\int_{0}^{L}p(x)\,v_{xx}\,\overline{u}_{xx}\ dx }\\
&  & + \int_{0}^{L}[-2\,q(x)\,v - (p(x)\,u_{xx})_{xx} -
\kappa\,\theta_{xx}]\,\overline{v}\ dx + \int_{0}^{L}(\kappa\,v_{x} +
\eta\,\theta_{x})_{x}
\,\overline{\theta}\ dx
\\
%& = & \int_{0}^{L}p(x)\,(v_{xx}\,\overline{u}_{xx} -
%\overline{v_{xx}\,\overline{u}_{xx}})\ dx -
%2\int_{0}^{L}q(x)\,|v|^{2}\ dx \\
%&  & +\ \kappa\int_{0}^{L}(\overline{v}_{x}\,\theta_{x} -
%\overline{\overline{v}_{x}\,\theta_{x}})\ dx -
%\eta\int_{0}^{L}|\theta_{x}|^{2}\,dx \\
& = & 2\,i\,Im\int_{0}^{L}p(x)\,v_{xx}\,\overline{u}_{xx}\ dx -
2\int_{0}^{L}q(x)\,|v|^{2}\ dx \\
&  & +\ 2\,i\,\kappa\,Im\int_{0}^{L}\overline{v}_{x}\,\theta_{x}\ dx -
\eta\int_{0}^{L}|\theta_{x}|^{2}\ dx .
\end{eqnarray*}
Taking real parts we obtain
\begin{equation}
\label{203}Re\left\langle\mathcal{A}U,\,U
\right\rangle_{\mathcal{H}} = -\
2\,\|\sqrt{q(x)}\,v\|_{L^{2}(0,\,L)}^{2} -
\eta\,\|\theta_{x}\|_{L^{2}(0,\,L)}^{2} .
\end{equation}
Thus $\mathcal{A}$ is a dissipative operator. Next, the domain $\mathcal{D} (\mathcal{A})$ is clearly dense in $\mathcal{H}$ and the operator is closed. 
Finally, let us show that $0\in\rho({\cal A})$, the resolvent of $\mathcal{A}$.  In fact, given
$F=(f,\,g,\,h)\in {\cal H},$ we must show that there exists a unique
$U=(u,\,v,\,\theta)$ in ${\cal D}({\cal A})$ such that ${\cal
A}U=F$, that is to say:
\begin{eqnarray}
\label{204}&  & v=f\quad \mbox{in}\quad H_{0}^{2}(0,\,L) \\
\label{205}&  & -\ 2\,q(x)\,v - (p(x)\,u_{xx})_{xx} -
\kappa\,\theta_{xx} =
g\quad \mbox{in}\quad L^{2}(0,\,L) \\
\label{206}&  & \kappa\,v_{xx} + \eta\,\theta_{xx}  = h\quad \mbox{in}\quad
L^{2}(0,\,L).
\end{eqnarray}
Taking $v=f$  in \eqref{206} we have
\begin{eqnarray}
\label{207}\theta_{xx}  \  = \  \frac{1}{\eta} \left( h - \kappa\,f_{xx}\right) \ \in \ L^{2}(0,\,L).
\end{eqnarray}
It is known that there is a unique $\theta\in H_{0}^{2}(0,\,L)$
satisfying \eqref{207} and
\begin{eqnarray*}
\|\theta_{xx}\|_{L^{2}(0,\,L)}\leq \, \frac{1}{\eta} \, \|h -
\kappa\,f_{xx}\|_{L^{2}(0,\,L)}\, \leq \, C\,\|F\|_{{\cal H}}
\end{eqnarray*}
for a positive constant $C.$ Moreover,  we have that $(\kappa\,v +
\eta\,\theta)_{xx} = h\in L^{2}(0,\,L)$ and we conclude that
$(\kappa\,v + \eta\,\theta)\in H_{0}^{2}(0,\,L).$ \\
Replacing \eqref{204} into \eqref{205} we have
\begin{eqnarray}
\label{208}&  & (p(x)\,u_{xx})_{xx}  = -\ g - 2\,q(x)\,f -
\kappa\,\theta_{xx} \quad \mbox{in}\quad L^{2}(0,\,L).
\end{eqnarray}
Moreover
\begin{eqnarray*}
\int_{0}^{L}u_{xx}^{2}\ dx =\int_{0}^{L}\frac{1}{p(x)}\
p(x)\,u_{xx}^{2}\ dx\leq \frac{1}{\alpha_{1}}
\int_{0}^{L}p(x)\,u_{xx}^{2}\ dx
=\frac{1}{\alpha_{1}}\,\|\sqrt{p(x)}\,u_{xx}\|_{L^{2}(0,\,L)}^{2}.
\end{eqnarray*}
It follows using the Lax-Milgram theorem that there exists a unique
function $u\in H_{0}^{2}(0,\,L)$ such that \eqref{208} holds in $L^{2}(0,\,L)$.
It is easy to show that $\|U\|_{{\cal H}}\leq C\,\|F\|_{{\cal H}}$
for a positive constant $C.$ Therefore we conclude that
$0\in\rho({\cal A}).$
 Therefore $\mathcal{A}$ is the generator of a $C^0$ semigroup of contractions and the proof of Theorem \ref{teo1} is complete. 
 \end{proof}

\renewcommand{\theequation}{\thesection.\arabic{equation}}
\setcounter{equation}{0}

\section{Asymptotic Behavior}\label{sec:expo}

In this section we show that the energy decays uniformly exponentially with time. We state our main result as follows. 
\begin{theorem}\label{teo2}
Let $u,\,$ $\theta$ be solution to system
\eqref{101}-\eqref{104}.   Then there exist
positive constants $M$ and $m$ such that
\begin{eqnarray*}
{\cal E}(t) \leq M\,{\cal E}(0)\,e^{-\,m\,t},\qquad \forall\;t\geq
0,
\end{eqnarray*}
where $\mathcal{E}$ is the total energy of the system, given by \eqref{201}.
\end{theorem}
We will prove this theorem using multiplier technique and a suitable Lyapunov.
We begin with preliminaries.

\subsection{Preliminaries}

We recall the energy equality from Lemma \ref{lemma:energy-equal}, 
\begin{equation}
\label{502} \frac{d}{dt}\mathcal{E}(t) = -\
2\int_{0}^{L}q(x)\,u_{t}^{2}\ dx - \eta\int_{0}^{L}\theta_{x}^{2}\ dx.
\end{equation}
Let us now define an auxiliary function that will help building our Lyapunov:
\begin{align}\label{def-F1}
 {\cal F}_{1}(t) = \int_{0}^{L}u\,u_{t}\ dx +
\int_{0}^{L}q(x)\,u^{2}\ dx.
\end{align}
%and define also:
%\begin{align}
%\label{def-F2} {\cal F}_{2}(t) = \int_{0}^{L}\theta\,u\ dx +
%\frac{1}{2}\,\kappa\int_{0}^{L}u_{x}^{2}\ dx.
%\end{align}
 We conclude this subsection by a lemma to control the time derivatives of the auxiliary function $\mathcal{F}_1$.
\begin{lemma}\label{lem:esti-F1}
Let $u, \theta$ be a smooth solution to \eqref{101}-\eqref{104}. Then the auxiliary function $\mathcal{F}_1$ defined by \eqref{def-F1} satisfies the following estimates. There exists $\mu > 0$ such that for all $t > 0:$
\begin{equation}\label{esti-F1}
 - \mu \ \mathcal{E} (t) \ \leq \, \mathcal{F}_1 (t) \, \leq \, \mu \, \mathcal{E} (t).
 \end{equation}
Moreover, the derivative of $\mathcal{F}_1$ verifies:
\begin{equation}\label{esti-F1prime}
\frac{d}{dt} \mathcal{F}_1 (t) \, = \, - 2 \, \mathcal{E} (t) + \mathcal{R}_1(t). 
\end{equation}
with the remainder $\mathcal{R}_1$ such that: there exist non negative constants $C_1, \, C_2, \, C_3 \, > 0$, such that, for all $\alpha > 0$, for any $t > 0$, it holds:
\begin{equation}\label{esti-R1}
\mathcal{R}_1 (t)  \ \leq \,   C_1 \, \int u_{t}^2 \, + C_2 \left( 1 + \frac 1 \alpha \right) \, \int \theta_x^2  + C_3 \, \alpha \, \int p u_{xx}^2.
\end{equation}
\end{lemma}

\begin{proof}
We begin with the first estimate. From the definition of $\mathcal{F}_1$, \eqref{def-F1}, 
On the other hand, using the classical Young and Poincar\'{e} inequalities successively, together with  the assumption on functions $p$ and $q$, \eqref{007}
we obtain
\begin{equation}\label{510}
\vert {\cal F}_{1}(t) \vert  
%& = & \int_{0}^{L}u\,u_{t}\ dx +
%\int_{0}^{L}q(x)\,u^{2}\ dx \nonumber  \\
%& \leq & \frac{L}{\pi}\int_{0}^{L}2\,u_{t}\,\frac{\pi}{2\,L}\,u\ dx
%+ C_{4}\int_{0}^{L}u^{2}\ dx \nonumber    \\
\,  \leq  \, \frac{1}{2} \,   \int_{0}^{L}\vert u_{t}\vert^{2}\ dx +   C (C_p, \alpha_1, \alpha_4) \, 
\int_{0}^{L}p(x)\,u_{xx}^{2}\ dx ,
\end{equation}
where $C (C_p, \alpha_1, \alpha_4)$ is a non negative constant, depending on Poincar\'e constants, and the upper and lower bounds of $p$ and $p$. The estimate \eqref{esti-F1} follows straightforward with  
\begin{equation*}
\mu \, = \,  \frac 1 2 + C (C_p, \alpha_1, \alpha_4).  
\end{equation*}
Next, we differentiate $\mathcal{F}_1$: by using  \eqref{101} and integrating by parts we have
\begin{eqnarray*}
\lefteqn{\frac{d}{dt}{\cal F}_{1}(t) = \int_{0}^{L}u\,u_{tt}\ dx +
\int_{0}^{L}u_{t}^{2}\ dx +
\frac{d}{dt}\int_{0}^{L}q(x)\,u^{2}\ dx } \nonumber \\
& = & \int_{0}^{L}u\left(-\ 2\,q(x)\,u_{t} - (p(x)\,u_{xx}) -
\kappa\,\theta_{xx}\right)dx + \int_{0}^{L}u_{t}^{2}\ dx +
\frac{d}{dt}\int_{0}^{L}q(x)\,u^{2}\ dx \nonumber  \\
& = & -\int_{0}^{L}p(x)\,u_{xx}^{2}\ dx -
\kappa\int_{0}^{L}\theta\,u_{xx}\ dx + \int_{0}^{L}u_{t}^{2}\ dx
\nonumber \\
& =  & - 2 \mathcal{E} + \mathcal{R}_1, 
\nonumber 
\end{eqnarray*}
with 
$$\mathcal{R}_1 \,  = \, 2 \, \int_0^L \, u_t^2 + \int_0^L \ \theta^2 - \kappa \, \int_0^L \, \theta \, u_{xx}. $$
Hence, using the Poincar\'{e} inequality, \eqref{007} and the generalized Young's inequality, we have, for all $\alpha > 0$, for all $t >0$:
\begin{eqnarray*}
\vert {\cal R}_{1}(t) \vert &\leq & 2 \, \int_0^L \, u_t^2 + C_p \ \int_0^L \ \theta_x^2 + \frac{C_p}{2 \, \alpha} \, \int_0^L \ \theta_x^2  + \frac{1}{2 \, \alpha_1}\, \alpha \, \int_0^L \, p(x)\, u_{xx}^2.
\end{eqnarray*}
Estimate \eqref{esti-R1} follows directly. We will chose a particular $\alpha > 0$ in the following. 
\end{proof}

With this Lemma, we are now ready to prove the exponential stability.

\subsection{Proof of Theorem \ref{teo2}.}

We  introduce a Lyapunov functional:
\begin{equation}
\label{def-G}{\cal G}(t) = {\cal E}(t) + \varepsilon \, \left({\cal F}_1(t) + \mathcal{F}_2  (t)\right),
\end{equation}
where $\varepsilon > 0$ is a positive constant to be adjusted later (as the $\alpha > 0$ of the previous Lemma).

Let us compute the derivative of $\mathcal{G}$. From \eqref{502}, together with \eqref{esti-F1prime} and \eqref{esti-R1} from Lemma \ref{lem:esti-F1}, we get:

\begin{equation}\label{509}
\begin{array}{lcl}
\frac{d}{dt}{\cal G}(t) & = &  -\ 2\int_{0}^{L}q(x)\,u_{t}^{2}\ dx
- \eta\int_{0}^{L}\theta_{x}^{2}\ dx    - 2 \ \varepsilon \, \mathcal{E}  (t)  + \varepsilon \, \mathcal{R}_1 (t)    \\
&&
\\
& \leq & -\ \left(2 - \varepsilon \, \frac{C_1}{\alpha_3}\right) \, \int_{0}^{L}q(x)\,u_{t}^{2}\ dx -
 \left(\eta - \varepsilon \, C_{2}\left(1 + \frac 1 \alpha \right) \right) \, \int_{0}^{L}\theta_{x}^{2}\ dx
  \\
  &&
  \\
&  & -  \varepsilon\,  \left(\ 2 \, \mathcal{E} (t)  - \, \alpha \, C_3 \, \int_0^L \, p u_{xx}^2 \right).  
\end{array}
\end{equation}
Now we will chose the constants $\varepsilon$ and $\alpha$. We need to control the last term in \eqref{509} in order to have a good differential inequality. We also need to ensure the positivity of the Lyapunov $\mathcal{G}$.

On the one hand, from the definition of the total energy, \eqref{201}, it holds:
$$  2 \mathcal{E} (t)  - C_3 \, \alpha \,  \int_0^L \, p u_{xx}^2  \, = \,  \left(1 - \alpha \, C_3 \right) \, \int_0^L \, p u_{xx}^2  + \int_{0}^{L}u_{t}^{2}\ dx +
\int_{0}^{L}|\theta|^{2}\ dx.$$
Hence, choosing $\alpha > 0 $ small enough so that 
$$  K \, : = \, 1 - \alpha \, C_3 \ > \ 0,$$
 we get, injecting in \eqref{509}:
 \begin{equation*}
\frac{d}{dt}{\cal G}(t) \, \leq \, - \ 2 \, K \, \varepsilon \, \mathcal{E} (t)   - \  K_1 \, \int_{0}^{L}q(x)\,u_{t}^{2}\ dx -  \ K_2  \, \int_{0}^{L}\theta_{x}^{2}\ dx
\end{equation*}
where $K > 0$ is defined above, while  $\alpha$ is now fixed to ensure $K > 0$, and:
$$K _1 \ = \ \left(2 - \varepsilon \, \frac{C_1}{\alpha_3}\right), \, \, \, K_2 \ = \  \eta - \varepsilon \, C_{2}\left(1 + \frac 1 \alpha \right).$$

On the other hand, from the definition of $\mathcal{G}$, \eqref{def-G}, and from \eqref{esti-F1} in Lemma \ref{lem:esti-F1}, we have, for all $t > 0$:
\begin{equation}\label{esti-G}
 \left( 1 - \varepsilon \ \mu \right) \, \mathcal{E} (t) \, \leq \,  \mathcal{G} (t) \, \leq \,  \left( 1 + \varepsilon \ \mu \right) \, \mathcal{E} (t). 
 \end{equation}
Therefore, it suffices now to choose $\varepsilon > 0$, small enough so that:
\begin{equation}\label{positiv-K}
\left\{ \begin{array}{l} 
1 - \varepsilon \ \mu  \ > \ 0,
\\
K _1 \ = \ \left(2 - \varepsilon \, \frac{C_1}{\alpha_3}\right) \ > \ 0, 
\\
K_2 \ = \  \eta - \varepsilon \, C_{2}\left(1 + \frac 1 \alpha \right) \ > \ 0.
\end{array}
\right. 
\end{equation}
 Hence, the positivity of the $K_i$, $i = 1, 2$  gives:
  \begin{equation}\label{esti-Gprime}
\frac{d}{dt}{\cal G}(t) \, \leq \, - \ 2 \, K \, \varepsilon \, \mathcal{E} (t),
\end{equation}
where $\varepsilon$ is now fixed, so that  \eqref{positiv-K}  is satisfied. 
Finally combining \eqref{esti-G} with \eqref{esti-Gprime} integrated in time, we get, for all $t > 0$:

  $$\mathcal{E} (t) \, \leq \,   M_0  - \frac{2 K \varepsilon}{ (1 - \varepsilon \, \mu) \, }  \, \int_0^t \mathcal{E} (s) d s,  $$
 with 
 $$M_0 \ = \ \frac{\mathcal{G} (0)}{1 - \varepsilon \, \mu} \ = \, \frac{1 + \varepsilon \, \mu}{1 - \varepsilon \, \mu} \, \mathcal{E} (0) \, \geq \  0.$$ 
 The conclusion  of Theorem $\ref{teo2}$, that is 
 $$\mathcal{E} (t) \, \leq \, M \, \mathcal{E} (0) \, e^{- m\, t}, $$ 
  follows immediately from the Gronwall Lemma, with:
 $$ m \ = \ \frac{2 \ K \ \varepsilon}{1 - \varepsilon \, \mu} \ > \ 0, \quad M \ = \ \frac{1 + \varepsilon \, \mu}{1 - \varepsilon \, \mu},$$
 and the proof is complete.

\section{Numerical simulation}\label{sec:numeric}

In this section we provide a numerical illustration of Theorem \ref{teo2}, that is the stabilization of the vibrations thanks to the thermal effect. Few numerical simulations have been don for a EB type equation, coupled with a damping and/or heat conduction. See for example \cite{FEM-stab-beam-Zhang}, where the authors use a Finite Element scheme in space, but do no care about the time discretization, since they are interested in the Optimal Control Model.  Here we want to illustrate our result of exponential stability for physically relevant initial conditions. Therefore, since we do not aim at designing a high order scheme, we consider a classical  second order Finite Difference scheme in both time and space (see for example \cite{smith1985}), taking care with the discretization of the bilaplacian operator ( \cite{kantorovitch, numeric-multiD-Liu-Sun, richardson}). 
We consider the following initial conditions: 

$$\left\{
\begin{array}{l}
u(x,0) \, = \, 0 , \, \, \forall \ x \in \ (0, L),
\\
\theta(x,0)\, =\,  A \, e^{- \frac{x^2}{\sigma^2}}, \, \, \forall \ x \in \ (0, L),
\end{array}
\right.
 $$
 with $A \ = \ 0.2$, whereas the standard deviation $\sigma^2$ changes from $10^{-4}$ to $10^{-2}$. 
 These initial conditions represent a possible physical experiment, don to test the response of a beam (the deformation) when we apply an initial heat pulse in the center of the beam, see \cite{wang-xu2001}.

 We plot in Figure \ref{fig1} the discrete energy (computed with discrete integration at all discrete times $t_n$, of the continuous expression \eqref{201}),
 versus time in $Log-Log$ scale, and observe the exponential decay for three values of the standard deviation.

\section*{Conclusion}
 In this work, we show  the exponential stabilization of an EB type model of  flexible inhomogeneous beam, taking into account frictional damping and heat conduction of Fourier type. We illustrate our result numerically, considering 
physically relevant initial conditions, and our result generalizes the results obtained for the wave equation and for EU models without heat. We think that this result is a first step toward a better understanding of more complicated models, such as the one derived in \cite{abouel-zenkour2015}.

 \newpage
 
 \section*{Acknowledgements}
Octavio Vera thanks the support of the Fondecyt project 1121120. 
We thank Centro CRHIAM Project Conicyt/Fondap-15130015 and Red Doctoral REDOC.CTA.
Amelie Rambaud thanks the support of Fondecyt project 11130378.

\begin{figure}[ht]\label{fig1}
\begin{center}
\includegraphics[width=6cm]{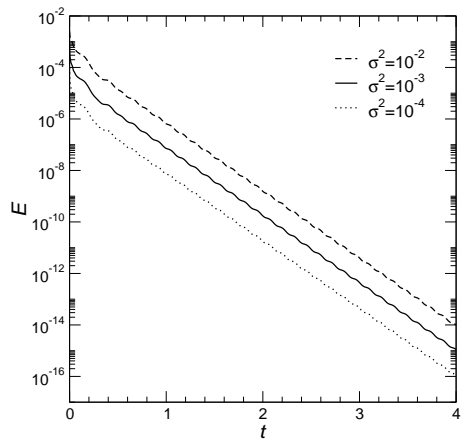}
\caption{Exponential decay}
\end{center}
\end{figure}

\end{document}